\documentclass[12pt]{amsart}
\usepackage{fullpage,verbatim,amssymb}
\usepackage{hyperref}
\hypersetup{pdfstartview={XYZ null null 1.00}}
\usepackage{enumerate}
\usepackage[usenames,dvipsnames]{color}
\definecolor{blue}{rgb}{0,0,1}
\definecolor{red}{rgb}{1,0,0}
\definecolor{green}{rgb}{0,.6,.2}
\definecolor{purple}{rgb}{1,0,1}

\long\def\red#1\endred{\textcolor{red}{#1}}
\long\def\blue#1\endblue{\textcolor{blue}{#1}}
\long\def\purple#1\endpurple{\textcolor{purple}{ #1}}
\long\def\green#1\endgreen{\textcolor{green}{#1}}

\newcommand{\bsl}{{\backslash}}

\newcommand{\Z}{\mathbb{Z}}
\newcommand{\Q}{\mathbb{Q}}
\newcommand{\R}{\mathbb{R}}
\newcommand{\T}{\mathbb{T}}
\newcommand{\C}{\mathbb{C}}

\newcommand{\sm}{\left(\begin{smallmatrix}}
\newcommand{\esm}{\end{smallmatrix}\right)}
\newcommand{\bpm}{\begin{pmatrix}}
\newcommand{\epm}{\end{pmatrix}}

\renewcommand{\epsilon}{\varepsilon}

\DeclareMathOperator{\SL}{SL}

\DeclareMathOperator{\M}{M}

\DeclareMathOperator{\ASL}{ASL}

\DeclareMathOperator{\Cb}{C}
\DeclareMathOperator{\Hb}{H}

\newtheorem{theorem}{Theorem}[section]

\newtheorem{lemma}{Lemma}[section]

\newtheorem{proposition}{Proposition}[section]

\newtheorem{corollary}{Corollary}[section]

\newtheorem{definition}{Definition}[section]

\theoremstyle{remark}

\newtheorem{remark}{Remark}[section]

\numberwithin{theorem}{section}
\numberwithin{equation}{section}

\def\veca{{\text{\boldmath$a$}}}

\def\vecb{{\text{\boldmath$b$}}}

\def\vecc{{\text{\boldmath$c$}}}

\def\vece{{\text{\boldmath$e$}}}

\def\veck{{\text{\boldmath$k$}}}
\def\vecl{{\text{\boldmath$l$}}}

\def\vecm{{\text{\boldmath$m$}}}
\def\vecn{{\text{\boldmath$n$}}}

\def\vecr{{\text{\boldmath$r$}}}

\def\vecs{{\text{\boldmath$s$}}}

\def\vect{{\text{\boldmath$t$}}}

\def\vecv{{\text{\boldmath$v$}}}

\def\vecx{{\text{\boldmath$x$}}}

\def\vecy{{\text{\boldmath$y$}}}

\def\vecnull{{\text{\boldmath$0$}}}

\def\i{\mathrm{i}}

\def\scrA{{\mathcal A}}
\def\scrB{{\mathcal B}}

\def\scrE{{\mathcal E}}

\def\cR{{\mathcal R}}
\def\scrS{{\mathcal S}}

\def\fg{{\mathfrak g}}

\def\fP{{\mathfrak P}}

\def\i{\mathrm{i}}

\def\diag{\operatorname{diag}}

\def\diam{\operatorname{diam}}

\def\C{\operatorname{C{}}}
\def\G{\operatorname{G{}}}

\def\M{\operatorname{M{}}}

\def\GL{\operatorname{GL}}

\def\SL{\operatorname{SL}}
\def\ASL{\operatorname{ASL}}

\def\Prob{\operatorname{Prob}}

\def\trans{\,^\mathrm{t}\!}

\makeatletter
\def\@tocline#1#2#3#4#5#6#7{\relax
  \ifnum #1>\c@tocdepth 
  \else
    \par \addpenalty\@secpenalty\addvspace{#2}%
    \begingroup \hyphenpenalty\@M
    \@ifempty{#4}{%
      \@tempdima\csname r@tocindent\number#1\endcsname\relax
    }{%
      \@tempdima#4\relax
    }%
    \parindent\z@ \leftskip#3\relax \advance\leftskip\@tempdima\relax
    \rightskip\@pnumwidth plus4em \parfillskip-\@pnumwidth
    #5\leavevmode\hskip-\@tempdima
      \ifcase #1
       \or\or \hskip 1em \or \hskip 2em \else \hskip 3em \fi%
      #6\nobreak\relax
    \hfill\hbox to\@pnumwidth{\@tocpagenum{#7}}\par
    \nobreak
    \endgroup
  \fi}
\makeatother

\title{Effective joint equidistribution of primitive rational points on expanding horospheres}
\date{\today}
\author{Daniel El-Baz, Bingrong Huang, and Min Lee}
\address{Institute of Analysis and Number Theory \\ Graz University of Technology \\ Steyrergasse 30, 8010 Graz \\ Austria}
\email{danielelbaz88@gmail.com}

\address{Data Science Institute and School of Mathematics  \\ Shandong University \\ Jinan \\ Shandong 250100 \\China}
\email{brhuang@sdu.edu.cn}

\address{School of Mathematics, University of Bristol, Bristol BS8 1TW, U.K.}
\email{\tt min.lee@bristol.ac.uk}

\begin{document}

\begin{abstract}
We prove an effective version of a result due to Einsiedler, Mozes, Shah and Shapira who established the equidistribution of primitive rational points on  expanding horospheres in the space of unimodular lattices in at least $3$ dimensions. Their proof uses techniques from homogeneous dynamics and relies in particular on measure-classification theorems --- an approach which does not lend itself to effective bounds. We implement a strategy based on spectral theory, Fourier analysis and Weil's bound for Kloosterman sums in order to quantify the rate of equidistribution for a specific horospherical subgroup in any dimension.
We apply our result to provide a rate of convergence to the limiting distribution for the appropriately rescaled diameters of random circulant graphs.
\end{abstract}

\maketitle
\tableofcontents

\section{Introduction}
In recent years, there has been an increased focus on obtaining effective versions of equidistribution theorems in homogeneous dynamics.
For the method it introduced, we single out Str\"ombergsson's breakthrough paper \cite{Stromb2015} and mention the related work by Browning and Vinogradov \cite{BV2016}.
Particularly interesting targets, of which these two papers are instances, consist of results whose proof relies on rigidity theorems such as Ratner's, which are by nature not effective.
The primary purpose of this paper is to accomplish this to get an effective version of a result due to Einsiedler, Mozes, Shah and Shapira \cite{EMSS2016}.
Their theorem was a conjecture due to Marklof, who had been able to prove an averaged version thereof and made great use of it \cite{Marklof2010}.
His proof relied on the mixing property of a certain diagonal flow on the space of unimodular lattices and was made effective, using estimates on the decay of matrix coefficients, by Li \cite{Li2015} who applied it to obtain a quantitative version of Marklof's result concerning the distribution of Frobenius numbers.
An article by Marklof and the third author \cite{LeeMarklof2017} provided a rate of convergence for the Einsiedler--Mozes--Shah--Shapira result for a certain horospherical subgroup in the two-dimensional setting according to the set-up below.
We now state our main result, which yields such a rate in any dimension for certain horospherical subgroups.


For $d\ge 1$, let $\Gamma = \SL_{d+1}(\Z)$ and define
\begin{equation}\Hb = \left\{ \bpm A & \vecv \\  \trans \vecnull & 1 \epm \, : \,  A \in \SL_d(\R), \vecv \in \R^d \right \} \subset \SL_{d+1}(\R).
\end{equation}
Denote by $\mu_{\Hb}$ the ${\Hb}$-invariant Haar probability measure on $\Gamma \backslash \Gamma {\Hb}$.
Finally, for $\vecx \in \R^d$, define
\begin{equation}
n_+(\vecx) = \bpm I_d & \vecnull \\ \trans \vecx & 1 \epm \in \SL_{d+1}(\R).
\end{equation}
We note that the group of all matrices of this form is the expanding horospherical subgroup corresponding to the semigroup of matrices of the form $\diag(e^t, \ldots, e^t, e^{-dt}) \in \SL_{d+1}(\R)$ with $t > 0$.

Define, for every positive integer $q$,
\begin{equation} D(q) = \diag(q^{\frac 1d}, \ldots, q^{\frac 1d}, q^{-1}) \in \SL_{d+1}(\R). \end{equation}
Note that for every $\vecr\in \Z^d$ satisfying $\gcd(q, \vecr)=1$, we have by \cite[Lemma 2.1]{EMSS2016} (see also \cite[Remark 3.3, (3.53)]{Marklof2010} and \cite[Lemma 4.1]{Li2015})
\begin{equation}\label{e:emss_lem}
\Gamma n_+(q^{-1} \vecr) D(q) \in \Gamma \bsl \Gamma {\Hb}.
\end{equation}
We also provide an explicit proof in \autoref{lem:emss_parameter}.
We define
\begin{equation}
\cR_q = \{ \vecr \in (\Z \cap [1,q])^d \, : \, \gcd(\vecr, q) = 1 \}.
\end{equation}

Let $\T\cong \Z\bsl \R$ be the unit circle and
$\Cb_b^k(\Gamma\backslash \Gamma {\Hb}\times \T^d)$ be the space of $k$ times continuously differentiable functions with all derivatives bounded and denote by $\|\cdot\|_{\Cb_b^k}$ the Sobolev norm (see \eqref{e:fCbk}).
Our main result is the following theorem.
\begin{theorem} \label{thm:main}
For every $d \ge 3$, every $\epsilon>0$ and every integer  $k \geq 2d^2-d+1$, there exists a constant $c > 0$ such that for every function $f \in \Cb_b^{k}(\Gamma \backslash \Gamma {\Hb} \times \T^d)$ and every $q \in \Z_{\ge 1}$,
\begin{equation}
  \left| \frac 1{\#\cR_q} \sum_{\vecr \in \cR_q} f \left(\Gamma n_+\left(\frac 1q \vecr \right) D(q), \frac 1q \vecr \right) - \int_{\Gamma \backslash \Gamma {\Hb} \times \T^d} f d \mu_{\Hb} d \vecx \right|
  \le c \| f \|_{\Cb_b^{k}} q^{-\frac 12 + \frac{d^2(2k-2d+1)}{2k^2} +\epsilon}.
\end{equation}
\end{theorem}

\begin{remark}
For $d=1$, this result was already known to Marklof in an effective form, with rate $O_f(q^{-\frac 12 + \epsilon})$  \cite{Marklof2010horospheres}. See also \cite[section 2.1]{EMSS2016} for a more detailed presentation of the argument. We merely mention that it relies on Weil's bound for Kloosterman sums as well, but is otherwise much simpler.
\end{remark}

\begin{remark} \label{rem:d=2}
We note that our proof also works when $d=2$ and hence recovers the previous result by Marklof and the third author \cite{LeeMarklof2017}. In this case, the error term becomes
 $c \|f\|_{\Cb_b^k} q^{-\frac 12 +\epsilon}(q^{\theta} + q^{\frac{2(2k-3)}{k^2}})$ for some constant $c>0$, 
where $\theta \ge 0$ is a Ramanujan bound for $\GL_2$ over $\Q$. The Ramanujan conjecture is the assertion that $\theta=0$ and the current record towards it is a result due to Kim and Sarnak which states that $\theta \le \frac 7{64}$, proved in \cite[Appendix 2]{KimSarnak2003}.
The reason for this discrepancy is that for $d \ge 3$, the use of bounds towards the Ramanujan conjecture for $\GL_d$ over $\Q$ can be bypassed.
Instead, Clozel, Oh and Ullmo \cite{ClozelOhUllmo2001} exploit the uniform version of Kazhdan's property $(T)$ for $\SL_d(\Q_p)$ for all primes $p$, when $d\geq 3$, as was obtained by Oh \cite{Oh2002}.
\end{remark}

As already hinted at, this result has several applications, for instance to the distribution of Frobenius numbers as in \cite{Marklof2010, Li2015} or to results about the shape of lattices as in \cite{EMSS2016}.
We highlight one in particular, which concerns the limiting distribution of the diameters of random Cayley graphs of $\Z/q\Z$ as $q \to +\infty$, following Marklof and Str\"ombergsson \cite{MarklofStrombergsson2013} (see also \cite{ShapiraZuck2019} for the case of random Cayley graphs of arbitrary finite abelian groups).
In \cite{AGG2010}, Amir and Gurel-Gurevich conjectured the existence of a limiting distribution, as $q \to +\infty$, for $\frac{\diam(q,d)}{q^{1/d}}$ where $\diam(q,d)$ denotes the diameter of the Cayley graph of $\Z/q\Z$ with respect to  the subset $\{\pm a_1, \ldots, \pm a_d \}$ where $(a_1, \ldots, a_d)$ is chosen uniformly at random from $\cR_q$.
Following the method expounded in \cite{MarklofStrombergsson2013}, the existence of this limiting distribution is a consequence of the main theorem in \cite{EMSS2016}.
By the same token, our \autoref{thm:main} implies the following result:
\begin{corollary} \label{cor:app}
For every $d \ge 3$, there exists a continuous non-increasing function $\Psi_d \colon \R_{\ge 0} \to \R_{\ge 0}$
with $\Psi_d(0)=1$ and $\lim_{R \to \infty} \Psi_d(R) = 0$,
and a constant $\eta_d > 0$
such that for every $\epsilon > 0$ and every $R \ge 0$, we have
\begin{equation}
\Prob\left(\frac{\diam(q,d)}{q^{1/d}} \ge R \right) = \Psi_d(R) + {O \left(q^{-\eta_d+\epsilon} \right)},
\end{equation}
where the implicit constant depends on $R$ and $\epsilon$.
\end{corollary}
We state a more precise version of the above corollary as \autoref{cor:app_err}, which also contains an explicit description of the limiting distribution in terms of the space of $d$-dimensional unimodular lattices.
At this point, we do however note that the decay of $\Psi_d$ as $R \to +\infty$ is known: it is proved in \cite[section 3.3]{MarklofStrombergsson2013} that for $d \ge 2$,
\begin{equation}
\Psi_d(R) = \frac 1{2 \zeta(d) R^d} + O_d \left(\frac 1{R^{d+1+\frac 1{d-1}}} \right).
\end{equation}
In order to deduce this corollary, which we do in \autoref{sec:app}, the explicit dependence on $f$ in the error term of \autoref{thm:main} is required.

Our strategy to prove \autoref{thm:main} is based on harmonic analysis and Weil's bound for Kloosterman sums, more precisely:
\begin{itemize}
\item in \autoref{sec:ratpts}, which contains the main novelty of our approach, we avoid the need to obtain an explicit solution to a (non-linear) system of equations modulo $q$ --- as was done for $d=2$ in \cite{LeeMarklof2017} --- by introducing a helpful parametrisation of $\cR_q$;
\item we then use Fourier analysis on the space of affine lattices in order to estimate the sum we are interested in --- this follows a strategy introduced by Str\"ombergsson in \cite{Stromb2015} for the space of shifted lattices in $2$ dimensions and we extend the required Fourier tools to any dimension in \autoref{sec:Fourier};
\item these estimates are carried out in \autoref{sec:main}: to get to the main term, the key ingredient is a deep result of Clozel, Oh and Ullmo \cite{ClozelOhUllmo2001}; to bound the error terms, we use estimates for Ramanujan and Kloosterman sums, combined with various counting arguments.
\end{itemize}

\noindent
{\bf Acknowledgements:}
We are very grateful to Jens Marklof, Hee Oh and Ze\'ev Rudnick for several insightful conversations and judicious comments on a previous version of this paper. 
We also thank the referees for their careful reading and several suggestions which led to a much-improved presentation.
The research of Daniel El-Baz and Bingrong Huang was supported by the European Research Council (ERC) under the European Union’s Horizon 2020 research and innovation programme (Grant agreement No. 786758). 
Daniel El-Baz is  supported by the Austrian Science Fund (FWF), projects F-5512 and Y-901.
Bingrong Huang is supported by NSFC (Nos. 12001314 and 12031008). 
Min Lee is supported by a Royal Society University Research Fellowship.

\section{Primitive rational points on horospheres} \label{sec:ratpts}
Let $d\geq 1$,  $\G = \SL_{d+1}(\R)$ and $\Gamma=\SL_{d+1}(\Z)$.
For any $g\in \G$, we write $g=\sm A& \vecb \\ \trans \vecc & D\esm$ where
$A\in \M_{d}(\R)$, $\vecb, \vecc\in \R^d$ and $D\in \R$.
Let $I_k$ be the $k\times k$ identity matrix.

For a positive integer $q$, we define the following congruence subgroup of $\SL_d(\Z)$:
\begin{equation}
\Gamma_{0, d}(q) = \left\{\gamma\in \SL_d(\Z)\; : \; \gamma\equiv \bpm * & * \\ \trans \vecnull & u\epm \pmod{q}\right\}.
\end{equation}
Note that for any $\gamma\in \Gamma_{0, d}(q)$ satisfying $\gamma\equiv \sm * & * \\ \trans\vecnull & u\esm\pmod{q}$,
$\gcd(u, q)=1$ holds.

We record the formula for the index of $\Gamma_{0,d}(q)$ inside $\SL_d(\Z)$.
\begin{proposition}\label{prop:Gamma0_index} For every $d \ge 2$ and $q \ge 1$, we have
\begin{equation}
[\SL_d(\Z) : \Gamma_{0,d}(q)] = q^{d-1} \prod_{p \mid q} \frac{1-p^{-d}}{1-p^{-1}}.
\end{equation}
\end{proposition}
\begin{proof}[Proof (sketch)]
It is a standard fact (for an explicit reference see, for instance,  \cite[Corollary 2.9]{Han2006}) that
\begin{equation}
\#\GL_d(\Z/q\Z) = q^{d^2} \prod_{p \mid q}  \left(1 - \frac 1{p^d} \right) \left(1 - \frac 1{p^{d-1}} \right) \cdots \left(1 - \frac 1p \right),
\end{equation}
from which it follows that
\begin{equation}
\#\SL_d(\Z/q\Z) = q^{d^2-1} \prod_{p \mid q} \left(1-\frac 1{p^d} \right) \left(1 - \frac 1{p^{d-1}} \right) \cdots \left(1 - \frac 1{p^2} \right).
\end{equation}
This last cardinality is precisely the index of the principal congruence subgroup \begin{equation} \Gamma_d(q) = \{ M \in \SL_d(\Z) \, : \, M \equiv I_d \pmod q \} \end{equation} inside $\SL_d(\Z)$ (for a reference about the surjectivity of the reduction map, see \cite[Proof of Lemma 1.38]{Shimura1971} for instance).
We note the inclusions $\Gamma_d(q) \subset \Gamma_{0,d}(q) \subset \SL_d(\Z)$ and therefore use the identity
\begin{equation} [\SL_d(\Z):\Gamma_d(q)] = [\SL_d(\Z):\Gamma_{0,d}(q)] [\Gamma_{0,d}(q):\Gamma_d(q)] \end{equation}
to conclude.
All that remains is to compute $[\Gamma_{0,d}(q):\Gamma_d(q)]$ and it is easy to see that it is equal to $q^{d-1} \#\GL_{d-1}(\Z/q\Z)$.
The desired formula follows.
\end{proof}


For a positive integer $q$, recall that
\begin{equation}\label{e:scrRq}
\cR_q = \left\{\vecr\in \left(\Z \cap (0, q] \right)^d \;:\; \gcd(\vecr, q) = 1 \right\}.
\end{equation}
We now give a simple formula and a lower bound for the size of this set.
\begin{lemma} For $d \ge 1$ and $q \ge 1$, we have
\begin{equation}
\# \cR_q
= \sum_{\delta\mid q}\mu(\delta) (q/\delta)^d.
\end{equation}
\end{lemma}
\begin{remark}
Note that when $d=1$, that is  $\varphi(q)$, Euler's totient function, as it should be.
\end{remark}
\begin{proof}
By partitioning all $d$-tuples $\vecr \in (\Z \cap [1,q])^d$ according to the value of $\gcd(\vecr, q)$, we see that
\begin{equation}
q^d = \sum_{\delta \mid q} \# \cR_{q/\delta}.
\end{equation}
The claim follows by M\"{o}bius inversion.
\end{proof}
We note the following trivial corollary.
\begin{corollary}
For $d\ge 1$ and $q \ge 1$,
\begin{equation} \label{e:RqEuler}
\# \cR_q = q^d \prod_{p \mid q} \left(1 - \frac 1{p^d} \right).
\end{equation}
In particular, for $d \ge 2$ and $q \ge 1$,
\begin{equation} \label{e:Rqlbd}
 \# \cR_q > \frac 1{\zeta(d)} q^d.
\end{equation}
\end{corollary}
\begin{remark}
The above inequality generalises \cite[(2.2)]{LeeMarklof2017}, whose proof has an unfortunate mistake (see the first inequality in \cite[(2.6)]{LeeMarklof2017}).
\end{remark}

Recall the following subgroup of $\G$:
\begin{equation}
{\Hb} = \bigg\{ \bpm A & \vecv \\ \trans \vecnull & 1\epm \;:\;  A\in \SL_{d}(\R), \vecv \in \R^{d}\bigg\}.
\end{equation}
For a positive integer $q$, we also recall that
\begin{equation}
	D(q) = \bpm q^{\frac{1}{d}} I_{d} & \vecnull \\ \trans \vecnull & q^{-1}\epm.
\end{equation}
By \eqref{e:emss_lem} we see that,
for every $\vecr\in \cR_q$,
there exist $A\in \SL_{d}(\R)$ and $\vecx\in \R^d$ such that
\begin{equation}\label{e:emss_matrix}
\Gamma n_+(q^{-1} \vecr) D(q) = \Gamma \bpm A & \vecx\\ \trans \vecnull & 1\epm.
\end{equation}
This is equivalent to the existence of $A\in \SL_d(\R)$ and $\vecx\in \R^d$, uniquely determined modulo $\Gamma$, satisfying
\begin{align}\label{e:modular_eq}
\bpm A & \vecx\\ \trans \vecnull & 1 \epm (n_+(q^{-1}\vecr) D(q))^{-1}
& =
\bpm A& \vecx\\ \trans \vecnull & 1\epm
\bpm q^{-\frac{1}{d}} I_d & \vecnull \\ \trans \vecnull & q\epm
\bpm I_{d} & \vecnull \\ -q^{-1 } \trans \vecr & 1\epm
\\ & =
\bpm \frac{q^{1-\frac{1}{d}} A -q\vecx \trans \vecr}{q} & q\vecx\\ -\trans \vecr & q\epm
\in \Gamma.
\nonumber
\end{align}
Let $\vecs = q\vecx$ and $B=q^{\frac{d-1}{d}}A$.
By the above relation,
\begin{equation}
\vecs\in \Z^d, \quad
\frac{1}{q}(B-\vecs \trans \vecr)\in \M_{d}(\Z)
\quad \text{ and }\quad
\det(B) = q^{d-1}\det(A) =q^{d-1}.
\end{equation}
So
\begin{equation}
B \in \M_d(\Z) \quad \text{ and } \quad B\equiv \vecs \trans \vecr \pmod{q}.
\end{equation}
Since
\begin{equation}\label{e:modular_eq_sol}
\bpm \frac{B-\vecs \trans \vecr}{q} & \vecs\\ -\trans \vecr & q\epm
\in \Gamma,
\end{equation}
we get that $\gcd(\vecs, q) = 1$ (see also \cite[Lemma~2.4]{EMSS2016}).

We now come to the goal of this section, which is to parametrise $\cR_q$ in terms of $\Gamma_{0, d}(q)\bsl \SL_d(\Z)$ and $(\Z/q\Z)^\times$.

Let $\mathcal B_q$ be a set of representatives for $\Gamma_{0,d}(q)\backslash \SL_d(\Z)$.
\begin{lemma}\label{lem:qprimitive_gamma}
We have
\begin{equation}\label{e:qprimitive_gamma}
\cR_q = \left\{\trans \gamma \bpm \vecnull \\ u\epm \pmod{q} \; :\; \gamma\in \mathcal B_q, \; u\in (\Z/q\Z)^\times\right\}.
\end{equation}
\end{lemma}
\begin{proof}

For any $\gamma\in \scrB_q$ and $u\in (\Z/q\Z)^\times$, there exists $\vecr\in (\Z\cap (0, q])^d$ such that
\begin{equation}
u\trans \gamma \vece_d = \trans\gamma \bpm \vecnull \\ u\epm \equiv \vecr \pmod{q},
\end{equation}
where $\vece_d$ is the last vector of the canonical basis of $\R^d$.
We claim that $\gcd(\vecr, q)=1$.
Let $\trans\veca$ be the last row of $\gamma$, that is, $\veca = \trans\gamma \vece_d$.
If $\gcd(\vecr, q)\neq 1$, since $u\veca-\vecr\equiv0\pmod{q}$,
this implies that $\gcd(u\veca, q)\neq 1$, so $\gcd(\veca)\neq 1$.
This contradicts the fact that $\gamma\in \SL_d(\Z)$.

 Note that, using \eqref{e:RqEuler} and \autoref{prop:Gamma0_index}, it follows that $\#\mathcal{R}_q = \# \mathcal{B}_q \cdot \varphi(q)$.
 Therefore, we only need to prove that $\trans\gamma u\vece_d \not\equiv \trans\gamma' u'\vece_d\pmod{q}$
if $(\gamma,u\pmod{q})\neq (\gamma',u' \pmod{q})$ for $\gamma, \gamma' \in \mathcal{B}_q$.
Indeed, suppose $\trans\gamma u\vece_d \equiv \trans\gamma'u' \vece_d\pmod{q}$.

Then
\[ \trans(\gamma (\gamma')^{-1})u\vece_d \equiv (\trans\gamma')^{-1} \trans\gamma u\vece_d\equiv u'\vece_d\pmod{q}, \]
that is, $\gamma (\gamma')^{-1}\in\Gamma_{0,d}(q)$.
Since $\gamma,\gamma'\in \mathcal{B}_q$, we get $\gamma=\gamma'$.
Using $\trans\gamma u\vece_d\equiv \trans\gamma'u'\vece_d\pmod{q}$
again, we obtain $u\equiv u'\pmod{q}$. This proves the lemma.
\end{proof}

Let $B_0 = \sm q I_{d-1} & \\ & 1\esm$; for every $\gamma\in \Gamma_{0, d}(q) \bsl \SL_d(\Z)$ and every $u \in (\Z/q\Z)^\times$,
if we set
\begin{align}
& \vecr\equiv u\trans\gamma \vece_d\pmod{q}, \label{e:parameter_r}\\
& \vecs = \overline{u} \vece_d, \; u\overline{u}\equiv 1\pmod{q}, \label{e:parameter_s}\\
& B= B_0 \gamma, \label{e:parameter_B}
\end{align}
then by \autoref{lem:qprimitive_gamma} we have $\vecr\in \cR_q$.
Moreover $\det(B)=q^{d-1}$ and $B\equiv \vecs \trans \vecr\pmod{q}$.
One checks that
\begin{equation}
\begin{pmatrix} q^{-1+\frac 1d} B & q^{-1} \vecs \\ ^t \vecnull & 1 \end{pmatrix} D(q)^{-1} n_+(-q^{-1}\vecr) \in \Gamma,
\end{equation}
which implies \eqref{e:emss_lem} and the following lemma, which is used in \S\ref{sec:main}.
\begin{lemma}\label{lem:emss_parameter}
For every $\gamma\in \Gamma_{0, d}(q) \bsl \SL_d(\Z)$ and $u\in (\Z/q\Z)^\times$, if we define $\vecr$, $\vecs$ and $B$
as in \eqref{e:parameter_r}-\eqref{e:parameter_B}, then
\begin{equation}\label{e:emss_parameter}
\Gamma n_+(q^{-1}\vecr) D(q)
= \Gamma \bpm q^{-1+\frac{1}{d}} B & q^{-1} \vecs \\ \trans \vecnull & 1\epm.
\end{equation}
\end{lemma}

\section{Fourier analysis on the space of lattice translates} \label{sec:Fourier}
In this section, we generalise the results given in \cite[section~4]{Stromb2015} and \cite[section~3]{LeeMarklof2017} to an arbitrary dimension.
When comparing with \cite{Stromb2015}, one should keep in mind that he uses a different representation for $\ASL_2(\R)$.

For $d\geq 2$, we define
\begin{equation}
\ASL_{d}(\R):= \SL_{d}(\R) \ltimes \R^d.
\end{equation}

For $M_1, M_2\in \SL_d(\R)$ and $\vecv_1, \vecv_2\in \R^d$,
the multiplication law on $\ASL_d(\R)$ is given by
\begin{equation}
(M_1, \vecv_1) \cdot (M_2, \vecv_2) = (M_1 M_2, M_1 \vecv_2 + \vecv_1).
\end{equation}
The discrete subgroup $\ASL_d(\Z)$ is defined similarly.

Let $\fg$ be the Lie algebra of $\ASL_d(\R)$, which we identify with $\mathfrak{sl}_d(\R) \oplus \R^d$.
We pick the following basis of $\fg$:
\begin{align}
& Y_{i, j} = (E_{i, j}, \vecnull), \quad 1 \leq i \ne j\leq d, \label{e:Yij}\\
& Y_i = (E_{i,i}-E_{1,1}, \vecnull), \quad i \ge 2, \label{e:Yi}\\
& X_i = (\vecnull, \vece_i), \quad 1\leq i \leq d, \label{e:Xi}
\end{align}
where $E_{i, j}\in \M_{d}(\R)$ has a $1$ at the $(i, j)$th entry and zeros elsewhere and the $\vece_i$ are the canonical basis of $\R^d$.

Each $(E, \vecy)\in \fg$ yields a left-invariant differential operator on a function on $\ASL_d(\R)$ in the following way:
\begin{equation}
((E, \vecy)F)(g, \vecx)
= \left.\frac{\partial}{\partial t} F((g, \vecx) \exp(tE, t\vecy)) \right|_{t=0}
= \left.\frac{\partial}{\partial t} F((g, \vecx) ((I_d, \vecnull)+t(E, \vecy))) \right|_{t=0}.
\end{equation}
In particular, for $X_{i_0}=(\vecnull, \vece_{i_0})$, $1\leq i_0 \leq d$, by the chain rule, we get
\begin{equation}\label{e:Ei0}
(X_{i_0} F)(g, \vecx)
= \sum_{i=1}^d g_{i, i_0} \left(\frac{\partial}{\partial x_{i}} F\right)(g, \vecx),
\end{equation}
where $g = (g_{i, j})_{1\leq i, j\leq d}$.

Let $\Cb_b^k(\ASL_d(\Z)\bsl \ASL_d(\R))$ denote the space of $k$ times continuously differentiable functions with all derivatives bounded.
For $F\in \Cb_b^k(\ASL_d(\Z)\bsl \ASL_d(\R))$ we set
\begin{equation}\label{e:FCbk}
\|F\|_{\Cb_b^k}
= \sum_{0\leq \ell \leq k} \sum_{\substack{E_i\in \{Y_{i_0, j_0}, X_{i_0}, Y_{i_0}\}\\ 1\leq i \leq \ell}}
\|E_1\circ \cdots\circ E_\ell F\|_{L^\infty}.
\end{equation}

Let $F$ be a function on $\ASL_{d}(\Z) \bsl \ASL_d(\R)$.
From now on, we implicitly identify functions on $\ASL_d(\Z) \bsl \ASL_d(\R)$
with $\ASL_d(\Z)$-invariant functions on $\ASL_d(\R)$.
For any $\vecm\in \Z^d$, \begin{equation}
F(A, \vecx+\vecm) = F((I_d, \vecm)(A, \vecx)) = F(A, \vecx)
\end{equation}
for $(A, \vecx)\in \ASL_{d}(\R)$.
So we have the following Fourier expansion of $F$:
\begin{equation}\label{e:fourier_F}
F(A, \vecx) = \sum_{\vecm\in \Z^d} \widehat{F}(A, \vecm) e^{2\pi \i \trans \vecm \vecx},
\end{equation}
where
\begin{equation}\label{e:wtF}
\widehat{F}(A, \vecm)
=
\int_{(\R/\Z)^d} F(A, \vect) e^{-2\pi \i \trans \vecm \vect}\; d \vect
\end{equation}
Here $d\vect$ denotes integration with respect to the Haar measure induced by the Lebesgue measure on $\R^d$.

\begin{lemma} \label{lem:FhatSLd}
For any $\gamma\in \SL_{d}(\Z)$ we have
\begin{equation}\label{e:whF_relation}
\widehat{F}(\gamma A, \vecm) = \widehat{F}(A, \trans \gamma \vecm).
\end{equation}
In particular, when $\vecm=\vecnull$, $\widehat{F}(A, \vecnull)$ is an automorphic function on $\SL_d(\Z)\bsl \SL_d(\R)$.
\end{lemma}
\begin{proof}
For any $\gamma\in \SL_d(\Z)$, we get
\begin{multline}
\widehat{F}(\gamma A, \vecm)
= \int_{(\R/\Z)^d} F(\gamma A, \vect) e^{-2\pi \i \trans \vecm \vect}\; d\vect
= \int_{(\R/\Z)^d} F((\gamma, \vecnull)(A, \gamma^{-1}\vect)) e^{-2\pi \i \trans \vecm \vect}\; d\vect
\\ =
\int_{(\R/\Z)^d} F(A, \vect) e^{-2\pi \i \trans (\trans \gamma \vecm) \vect}\; d\vect
= \widehat{F}(A, \trans\gamma \vecm).
\end{multline}
Here in the third identity we use the fact that $F$ is  left  $\ASL_d(\Z)$-invariant
and $\vect \mapsto \gamma \vect$ is a diffeomorphism of $(\R/\Z)^d$ preserving the volume measure $d\vect$.
\end{proof}

 Set $A=(a_{i, j})_{1\leq i, j\leq d}$.
For each $1\leq i_0 \leq d$, by applying integration by parts, we get
\begin{align}
\widehat{(X_{i_0}F)}(A, \vecm)
& = \int_{(\R/\Z)^d} (X_{i_0}F) (A, \vect) e^{-2\pi \i \trans \vecm \vect} \; d\vect
\\ & = \sum_{i=1}^d a_{i, i_0} \int_{(\R/\Z)^d} \frac{\partial}{\partial t_i} F(A, \vect) e^{-2\pi \i \trans \vecm \vect} \; d\vect
\nonumber \\ & =
\left(\sum_{i=1}^d a_{i, i_0} 2\pi \i m_i \right)\int_{(\R/\Z)^d} F(A, \vect) e^{-2\pi \i \trans \vecm \vect} \; d\vect
\nonumber \\ & = 2\pi \i \left(\sum_{i=1}^d m_i a_{i, i_0} \right) \widehat{F}(A, \vecm).
\nonumber
\end{align}
So for $k\in \Z_{\geq 1}$,
\begin{equation}
\int_{(\R/\Z)^d} (X_{i_0}^k F)(A, \vect) e^{-2\pi \i \trans \vecm \vect} \; d\vect
=
\left(2\pi \i \sum_{i=1}^d m_i a_{i, i_0} \right)^k \widehat{F}(A, \vecm),
\end{equation}
and we get
\begin{equation}
(2\pi)^k \left|\sum_{i=1}^d m_i a_{i, i_0} \right|^k \left|\widehat{F}(A, \vecm)\right|
\leq \int_{(\R/\Z)^d} \left|(X_{i_0}^k F)(A, \vect)\right| \; d\vect \leq \|X_{i_0}^k F\|_\infty.
\end{equation}
For $\vecb\in \R^d$, let $\|\vecb\|_\infty := \max_{\{1\leq i \leq d\}} \{|b_i|\}$.
Then
\begin{equation}
\max_{1\leq i_0\leq d} \left\{\left|\sum_{i=1}^d m_i a_{i, i_0} \right|^k\right\} = \|\trans A\vecm\|_{\infty}^k,
\end{equation}
and we have
\begin{equation}
(2\pi \|\trans A\vecm\|_{\infty})^k \left|\widehat{F}(A, \vecm)\right|
= \max_{1\leq i_0 \leq d} (2\pi)^k \left|\sum_{i=1}^d a_{i, i_0} m_{i}\right|^k
\left|\widehat{F}(A, \vecm)\right|
\leq \max_{1\leq i_0\leq d} \|X_{i_0}^k F\|_{\infty}
\leq \|F\|_{\Cb_b^k}.
\end{equation}
So for $\vecm\neq\vecnull$, we have
\begin{equation}\label{e:bound_FC_F}
\left|\widehat{F}(A, \vecm)\right| \leq \frac{\|F\|_{\Cb_b^k}}{(2\pi \|\trans A\vecm\|_{\infty})^k}.
\end{equation}

\section{Proof of the main theorem}\label{sec:main}
In this section, we prove \autoref{thm:main}.

For $f\in \Cb_b^k(\Gamma\bsl \Gamma {\Hb} \times (\R/\Z)^{d})$,
 using the fact that
$\Gamma \backslash \Gamma {\Hb}$ is diffeomorphic to $\ASL_d(\Z) \backslash \ASL_d(\R)$, we set, similarly to \eqref{e:FCbk},
\begin{equation}\label{e:fCbk}
\|f\|_{\Cb_b^k}
=\sum_{0\leq \ell \leq k} \sum_{\substack{E_i\in \{Y_{i_0, j_0}, X_{i_0}, Y_{i_0}\} \\ 1\leq i \leq \ell}}
\sum_{\substack{\ell_1, \ldots, \ell_d\geq 0, \\ \ell_1+\cdots +\ell_d +\ell \leq k}}
\left\|E_1\circ\cdots\circ E_\ell \frac{\partial^{\ell_1}}{\partial x_1^{\ell_1}} \cdots \frac{\partial^{\ell_d}}{\partial x_d^{\ell_d}}
f\right\|_{L^\infty}.
\end{equation}

We have the Fourier expansion
\begin{equation}
f(g, \vecx) = \sum_{\vecn\in \Z^d} \widehat{f_\vecn}(g) e^{2\pi \i \trans \vecn \vecx}
\end{equation}
where
\begin{equation}\label{eqn:hat-f}
\widehat{f_{\vecn}}(g)
= \int_{(\R/\Z)^d} f(g, \vecx) e^{-2\pi \i \trans \vecn \vecx} \; d\vecx.
\end{equation}
By using integration by parts repeatedly, for $\vecn\neq\vecnull$, we have
\begin{equation}\label{eqn:sup-f}
  \sup_{g\in \Gamma\backslash \Gamma {\Hb}} \Big|\widehat{f_{\vecn}}(g)\Big|
  \ll_k \|f\|_{\Cb_b^k} \|\vecn\|_\infty^{-k}.
\end{equation}

By  \autoref{lem:emss_parameter},
we have (recall that $\scrB_q$ is a set of representatives for $\Gamma_{0, d}(q)\bsl \SL_d(\Z)$):
\begin{align}
\frac 1{\#\cR_q} \sum_{\vecr \in \cR_q} & f \left(\Gamma n_+\left(\frac 1q \vecr \right) D(q), \frac 1q \vecr \right)
\\ & = \frac{1}{\#\cR_q} \sum_{\gamma\in \mathcal B_q}
\sum_{u\in (\Z/q\Z)^\times}
f\left(\bpm q^{-1+\frac{1}{d}} B_0 \gamma & q^{-1} \overline{u} \vece_d\\ \trans \vecnull & 1\epm,
q^{-1} u \trans\gamma \vece_d\right)
\nonumber \\ &= \frac{1}{\#\cR_q} \sum_{\gamma\in \mathcal B_q}
\sum_{u\in (\Z/q\Z)^\times} \sum_{\vecn\in \Z^d}
\widehat{f_\vecn} \left(\bpm q^{-1+\frac{1}{d}} B_0 \gamma & q^{-1} \overline{u} \vece_d\\ \trans \vecnull & 1\epm \right)
e^{2\pi \i \frac{ \trans \vecn u\trans\gamma \vece_d}{q}}.
\nonumber
\end{align}

We first note that we can truncate $\vecn$-sum at $\|\vecn\|_\infty\leq q^{\vartheta_1}$ for some small  $0<\vartheta_1<\frac{1}{2}$. Indeed, by \eqref{eqn:sup-f}, we know that the contribution from the terms with $\|\vecn\|_\infty > q^{\vartheta_1}$ is
\begin{equation}\label{eqn:error-n>}
  \ll_k \frac{\|f\|_{\Cb_b^k}}{\#\cR_q} \sum_{\gamma\in \mathcal B_q}
  \sum_{u\in (\Z/q\Z)^\times}
  \sum_{\substack{\vecn\in \Z^d\\ \|\vecn\|_\infty>q^{\vartheta_1}}}
  \frac 1{\|\vecn\|_\infty^k}
  \ll_k \|f\|_{\Cb_b^k}
  \sum_{\substack{\vecn\in \Z^d\\ \|\vecn\|_\infty>q^{\vartheta_1}}}
  \frac 1{\|\vecn\|_\infty^k}.
\end{equation}
Note that $\|\vecn\|_\infty \leq \|\vecn\|_2 \leq \sqrt{d}\|\vecn\|_\infty$.
 It is a standard fact that
\begin{equation} \label{Gauss}
  \sum_{\substack{\vecn\in \Z^d \\ \|\vecn\|_\infty > q^{{\vartheta_1}}}}
  \frac{1}{\|\vecn\|_\infty^{k}}
  \ll_{d,k}
  (q^{{\vartheta_1}})^{d-k}
  = q^{-{\vartheta_1}(k-d)}.
\end{equation}

So we have
\begin{multline}\label{eqn:sumr2sumn}
  \frac 1{\#\cR_q} \sum_{\vecr \in \cR_q} f \left(\Gamma n_+\left(\frac 1q \vecr \right) D(q), \frac 1q \vecr \right)
  \\
  =
  \frac{1}{\#\cR_q} \sum_{\gamma\in \mathcal B_q}
  \sum_{u\in (\Z/q\Z)^\times} \sum_{\substack{\vecn\in \Z^d \\ \|\vecn\|_\infty\leq q^{\vartheta_1}}}
  \widehat{f_\vecn} \left(\bpm q^{-1+\frac{1}{d}} B_0 \gamma & q^{-1} \overline{u} \vece_d\\ \trans \vecnull & 1\epm \right)
  e^{2\pi \i \frac{ \trans \vecn u\trans\gamma \vece_d}{q}} \\
  + O_{d, k}(\| f \|_{\Cb_b^{k}} q^{-\vartheta_1 (k-d)}).
\end{multline}

For $A\in \SL_d(\R)$ and $\vecy\in \R^d$, let
\begin{equation}\label{eqn:Fn-def}
F_\vecn(A, \vecy) = \widehat{f_{\vecn}}\left(\bpm A& \vecy\\ \trans \vecnull & 1\epm \right).
\end{equation}
Then $F_\vecn$ is a function on $\ASL_{d}(\Z) \bsl \ASL_{d}(\R)$
and as such has the following Fourier expansion
\begin{equation}\label{e:Fourier_F}
F_{\vecn}(A, \vecy) = \sum_{\vecm\in \Z^d} \widehat{F_{\vecn}}(A, \vecm) e^{2\pi \i \trans \vecm \vecy}.
\end{equation}
Here
\begin{equation}\label{eqn:hat-Fn}
\widehat{F_{\vecn}}(A, \vecm) = \int_{(\R/\Z)^d} F_{\vecn}(A, \vect) e^{-2\pi \i \trans \vecm \vect} \; d\vect.
\end{equation}
Recall that $B_0 = \sm q I_{d-1} & \\ & 1\esm$.
By \eqref{eqn:sumr2sumn} and \eqref{e:Fourier_F}, we get
\begin{multline}\label{eqn:sumr2sumnm}
\frac 1{\#\cR_q} \sum_{\vecr \in \cR_q} f \left(\Gamma n_+\left(\frac 1q \vecr \right) D(q), \frac 1q \vecr \right)
\\ =
\frac{1}{\#\cR_q} \sum_{\vecn\in \Z^d} \sum_{\vecm\in \Z^d} \sum_{\gamma\in \mathcal{B}_q}
\widehat{F_{\vecn}}\left(q^{-1+\frac{1}{d}} B_0 \gamma, \vecm\right)
\sum_{u\in (\Z/q\Z)^\times} e^{2\pi \i \frac{\trans \vecm \overline{u} \vece_d + \trans \vecn u\trans\gamma \vece_d}{q}}
\\ =
\frac{1}{\#\cR_q} \sum_{\substack{\vecn\in \Z^d \\ \|\vecn\|_\infty\leq q^{\vartheta_1}}} \sum_{\vecm\in \Z^d} \sum_{\gamma\in \mathcal{B}_q}
\widehat{F_{\vecn}}\left(q^{-1+\frac{1}{d}} B_0 \gamma, \vecm\right)
S(m_d, \trans \vecn \trans \gamma \vece_d; q)
 + O_{d, k}(\| f \|_{\Cb_b^{k}} q^{-\vartheta_1 (k-d)}).
\end{multline}
Here $S(a,b;q)=\sum_{u\in(\mathbb{Z}/q\mathbb{Z})^\times} e^{2\pi \i \frac{a \bar u + bu}{q}}$
is the classical Kloosterman sum.

Note that by \eqref{e:bound_FC_F},  for $\vecm\neq \vecnull$, we have
\begin{equation}\label{e:bound_hatFn}
\left|\widehat{F_\vecn}\left(q^{-1+\frac{1}{d}} B_0 \gamma, \vecm\right)\right|
\leq \frac{\|F\|_{\Cb_b^k}}{(2\pi \|q^{-1+\frac{1}{d}} \trans\gamma B_0 \vecm\|_{\infty})^k}.
\end{equation}

We bound \eqref{eqn:sumr2sumnm} by considering the following four distinct cases in the subsequent four propositions:
\begin{itemize}
\item $\vecm=\vecn=\vecnull$;
\item $\vecn\neq \vecnull$ and $\vecm=\vecnull$;
\item $\vecn, \vecm\neq \vecnull$ and $\trans\gamma q^{-1+\frac{1}{d}}B_0\vecm$ ``small";
\item $\vecn, \vecm\neq \vecnull$ and $\trans\gamma q^{-1+\frac{1}{d}} B_0 \vecm$ ``large".
\end{itemize}
We prove these propositions in the subsections below, from which \autoref{thm:main} follows as we explain after the statements of the propositions.

\begin{proposition}\label{prop:mainterm}
For every $\epsilon>0$  and every integer $k \geq d^2$, we have
\begin{multline}
\frac{1}{\#\cR_q} \sum_{\gamma\in \mathcal{B}_q}
\widehat{F_{\vecnull}}(q^{-1+\frac{1}{d}}B_0 \gamma, \vecnull)
S(0, 0; q)
\\ =
\int_{\SL_d(\Z) \bsl \SL_d(\R)} \widehat{F_{\vecnull}}(A, \vecnull) \; d\mu(A)
+ O_{\epsilon}(\|\widehat{F_{\vecnull}}(*, \vecnull) \|_{\Cb_b^k} q^{-\frac{1}{2}+\epsilon}).
\end{multline}
\end{proposition}

\begin{proposition}\label{prop:E1}
For each $\vecnull\neq \vecn\in \Z^d$ with $\|\vecn\|_\infty \leq q^{\vartheta_1}$, we have
\begin{equation}
\mathcal E_1
= \frac{1}{\#\cR_q}
\left| \sum_{\gamma\in \mathcal{B}_q}
\widehat{F_{\vecn}}\left(q^{-1+\frac{1}{d}} B_0 \gamma, \vecnull\right)
S(0, \trans\vecn\trans\gamma \vece_d; q)\right|
\\ \ll_\epsilon \|F\|_{\Cb_b^0} q^{-1+\vartheta_1+\epsilon}.
\end{equation}
\end{proposition}

\begin{proposition}\label{prop:E2}
For each $\vecn\in \Z^d$ with $\|\vecn\|_\infty \leq q^{\vartheta_1}$ and $0<\vartheta_2<\frac{1}{2d}$, we have
\begin{multline}
\mathcal{E}_2
=\frac{1}{\#\cR_q}
\left|\sum_{\substack{\vecm\in \Z^d\setminus\{\vecnull\}, \\ \|\trans \gamma q^{-1+\frac{1}{d}} B_0 \vecm\|_\infty \leq q^{\vartheta_2}}}
\sum_{\gamma\in \scrB_q}
\widehat{F_{\vecn}}\left(q^{-1+\frac{1}{d}} B_0 \gamma, \vecm\right)
S(m_d, \trans\vecn\trans\gamma \vece_d; q)\right|
\\ \leq \|F\|_{\Cb_b^0} q^{-\frac{1}{2}+ d\vartheta_2}
\frac{\sigma_{0}(q)^2}{\prod_{p\mid q}(1-p^{-1})}.
\end{multline}
 Here $\sigma_0(q)$ is the number of positive divisors of $q$.
\end{proposition}

\begin{proposition}\label{prop:E3}
For each $\vecn\in \Z^d$ with $\|\vecn\|_\infty \leq q^{\vartheta_1}$ and $0<\vartheta_2<\frac{1}{2d}$, we have
\begin{multline}
\mathcal E_3
= \frac{1}{\#\cR_q}
\left|\sum_{\substack{\vecm\in \Z^d\setminus\{\vecnull\}, \\ \|\trans \gamma q^{-1+\frac{1}{d}} B_0 \vecm\|_\infty > q^{\vartheta_2}}}
\sum_{\gamma\in \mathcal{B}_q}
\widehat{F_{\vecn}}\left(q^{-1+\frac{1}{d}} B_0 \gamma, \vecm\right)
S(m_d, \trans\vecn\trans\gamma \vece_d; q)\right|
\\ \ll_{d,k,\vartheta_2}  \|F\|_{\C_b^k} q^{-\frac{1}{2}+d\vartheta_2},
\end{multline}
  provided $k$ is an integer such that $k\geq \frac{2d-1}{2\vartheta_2}$.
\end{proposition}

\begin{proof}[Proof of \autoref{thm:main}]
By \eqref{eqn:sumr2sumnm} and Propositions \ref{prop:mainterm}--\ref{prop:E3}, we have
\begin{multline}
\frac 1{\#\cR_q} \sum_{\vecr \in \cR_q} f \left(\Gamma n_+\left(\frac 1q \vecr \right) D(q), \frac 1q \vecr \right)
= \int_{\SL_d(\Z) \bsl \SL_d(\R)} \widehat{F_{\vecnull}}(A, \vecnull) \; d\mu(A) \\
+ O(\| f \|_{\Cb_b^{k}} q^{-\vartheta_1 (k-d)})
+ O(\| f \|_{\Cb_b^{k}} q^{-\frac 12 + d (\vartheta_1 + \vartheta_2) +\epsilon}),
\end{multline}
for any $\epsilon>0$,  $0<\vartheta_1<\frac{1}{2}$, $0<\vartheta_2<\frac{1}{2d}$ and  $k\geq\max\{\frac{2d-1}{2\vartheta_2}, d^2\}$.
Note that by \eqref{eqn:hat-Fn}, \eqref{eqn:Fn-def}, and \eqref{eqn:hat-f}, we have
\begin{multline}
\int_{\SL_d(\Z) \bsl \SL_d(\R)} \widehat{F_{\vecnull}}(A, \vecnull) \; d\mu(A)
= \int_{\SL_d(\Z) \bsl \SL_d(\R)}
\int_{(\R/\Z)^d} \widehat{f_0}\left(\bpm A & \vect \\ \trans \vecnull & 1\epm \right) \; d\vect  \;d\mu(A)
\\ = \int_{\Gamma \backslash \Gamma {\Hb} \times \T^d} f d \mu_{\Hb} d \vecx.
\end{multline}
Taking $\vartheta_2=\frac{2d-1}{2k}$ and $\vartheta_1=\frac{1/2-d\vartheta_2}{k}$,
 we see that $\vartheta_2 = \frac{2d-1}{2k} < \frac{1}{2k}$ and $0< \vartheta_1 = \frac{1/2-d\vartheta_2}{k}$
if and only if $k\geq 2d^2-d+1$ (which is at least $d^2$, so that \autoref{prop:mainterm} applies).
This proves \autoref{thm:main}.
\end{proof}

\subsection{The main term: effective equidistribution of Hecke points}

Recall that $\scrB_q$ is a set of representatives for $\Gamma_{0, d}(q) \bsl \SL_d(\Z)$.
Throughout this section, we let $B_0 = \sm q I_{d-1} & \\ & 1\esm$.

\begin{lemma}
We have
\begin{equation}
\SL_d(\Z) B_0 \SL_d(\Z) = \underset{\delta\in {\scrB_q}}{\dot\bigcup} \SL_d(\Z) (B_0 \delta).
\end{equation}
\end{lemma}
\begin{proof}
We first check that the decomposition on the right hand side is disjoint.
For $\delta_1, \delta_2\in {\scrB_q}$,
if $\gamma B_0 \delta_1 = B_0 \delta_2$ for some $\gamma\in \SL_d(\Z)$,
then $B_0^{-1} \gamma B_0 = \delta_2 \delta_1^{-1}\in \SL_d(\Z)$.
Note that in this case
\begin{equation}
\delta_2 \delta_1^{-1}=B_0^{-1} \gamma B_0 = \bpm q^{-1} I_{d-1} & \\ & 1\epm \gamma \bpm q I_{d-1} & \\ & 1\epm
\in \Gamma_{0, d}(q).
\end{equation}
So we get $\delta_2\in \Gamma_{0, d}(q) \delta_1$.

From the construction, it is clear that
\begin{equation}
\SL_d(\Z) B_0 \SL_d(\Z) \supset \underset{\delta\in\scrB_q}{\dot\bigcup} \SL_d(\Z) (B_0\delta).
\end{equation}

Let $\tau = \sm T& \vect\\ \trans\vecs & t\esm \in \Gamma_{0, d}(q)$ for $\vecs\equiv\vecnull\pmod{q}$.

Then
\begin{equation}
B_0 \tau B_0^{-1} = \bpm q I_{d-1} & \\ & 1\epm \bpm T & \vect\\ \trans\vecs & t\epm \bpm q^{-1} I_{d-1} & \\ & 1\epm
= \bpm T & q\vect\\ q^{-1} \trans\vecs & t\epm
\in \SL_d(\Z),
\end{equation}
so $B_0 \Gamma_{0, d}(q) B_0^{-1} \subset \SL_d(\Z)$.
Take $\gamma_1, \gamma_2\in \SL_d(\Z)$.
There exists $\delta_2\in {\scrB_q}$ such that $\gamma_2\in \Gamma_{0, d}(q)\delta_2$.
We have
\begin{equation}
\gamma_1 B_0 \gamma_2 \in \gamma_1 B_0 \Gamma_{0, d}(q) \delta_2 = \gamma_1 (B_0 \Gamma_{0, d}(q) B_0^{-1}) B_0 \delta_2
\subset \SL_d(\Z) B_0 \delta_2
\end{equation}
and this implies that
\begin{equation}
\SL_d(\Z) B_0 \SL_d(\Z) \subset \underset{\delta\in \scrB_q}{\dot\bigcup} \SL_d(\Z) (B_0\delta),
\end{equation}
as claimed.
\end{proof}

Note that $\det(q^{-\frac{d-1}{d}} B_0)=1$ so $q^{-\frac{d-1}{d}} B_0\in \SL_d(\R)$.
For a function $F: \SL_d(\Z) \bsl \SL_d(\R) \to \mathbb C$, following \cite{ClozelOhUllmo2001},
the Hecke operator for $B_0$ is defined as
\begin{equation}\label{e:Heckeop}
(T_{B_0} F)(g) = \frac{1}{\#({\scrB_q})}
\sum_{\delta\in {\scrB_q}} F\left(q^{-\frac{d-1}{d}} B_0 \delta g\right).
\end{equation}
Assume that $F\in L^2(\SL_d(\Z) \bsl \SL_d(\R))$.
Following the argument in \cite[section~1]{ClozelOhUllmo2001} with \cite[Theorem~1.1]{ClozelOhUllmo2001} and the formula from \cite[p.~346]{ClozelOhUllmo2001},
we get
\begin{equation}\label{e:Hecke_L2}
\left\|T_{B_0} F - \int_{\SL_d(\Z) \bsl \SL_d(\R)} F(g) \; d\mu(g)\right\|_2
\ll_{\epsilon} q^{-\frac{1}{2}+\epsilon} \|F\|_2,
\end{equation}
for any $\epsilon>0$.
Note that the implicit constant only depends on $\epsilon$.
By \cite[Proposition~8.2]{ClozelUllmo2004}, we find that this $L^2$-convergence implies the same rate for the point-wise convergence: for every integer  $k\geq d^2$,
if $F\in \Cb_b^k(\ASL_d(\Z) \bsl \ASL_d(\R))$, then we get
\begin{equation}\label{e:Hecke_pointwise}
\left|T_{B_0}F(I_d) - \int_{\SL_d(\Z) \bsl \SL_d(\R)} F(g) \; d\mu(g) \right|
\ll_{\epsilon} q^{-\frac{1}{2}+\epsilon} \|F\|_{\Cb_b^k}.
\end{equation}

\begin{proof}[Proof of \autoref{prop:mainterm}]
Since $f$ is bounded, $\widehat{F_{\vecnull}}(*, \vecnull)\in L^2(\SL_d(\Z) \bsl \SL_d(\R))$, where the invariance under $\SL_d(\Z)$ follows from \autoref{lem:FhatSLd}.

For $S(0, 0; q) = \varphi(q)$, we get
\begin{equation}
\frac{1}{\#\cR_q} \sum_{\gamma\in {\scrB_q}}
\widehat{F_{\vecnull}}(q^{-1+\frac{1}{d}}B_0 \gamma, \vecnull)
S(0, 0;q)
= \frac{1}{\#\scrB_q}
\sum_{\gamma\in {\scrB_q}}
\widehat{F_{\vecnull}}(q^{-1+\frac{1}{d}}B_0 \gamma, \vecnull)
=T_{B_0} \widehat{F_{\vecnull}}(I_d, \vecnull).
\end{equation}
By \eqref{e:Hecke_pointwise},  for any integer $k\geq d^2$,
we get
\begin{equation}
\left|T_{B_0} \widehat{F_{\vecnull}}(I_d, \vecnull) - \int_{\SL_d(\Z) \bsl \SL_d(\R)} \widehat{F_{\vecnull}}(g, \vecnull) \; d\mu(g) \right|
\ll_{\epsilon} q^{-\frac{1}{2}+\epsilon} \|\widehat{F_{\vecnull}}(*, \vecnull)\|_{\Cb_b^k}.
\end{equation}
This completes the proof of \autoref{prop:mainterm}.
\end{proof}

\subsection{The first error term}

\begin{proof}[Proof of \autoref{prop:E1}]
Note that $\widehat{F_{\vecn}}\left(q^{-1+\frac{1}{d}} B_0 \gamma, \vecnull\right) \ll \|F\|_{\Cb_b^0}$ and
\begin{equation}
|S(0, \trans\vecn\trans\gamma \vece_d; q)|
= \left|\mu\left(\frac{q}{\gcd(q, \trans\vecn\trans\gamma \vece_d)}\right)\right|
\frac{\varphi(q)}{\varphi\left(\frac{q}{\gcd(q, \trans\vecn\trans\gamma\vece_d)}\right)}
\leq \gcd(q,\trans\vecn\trans\gamma \vece_d).
\end{equation}
The first equality holds since $S(0, \trans\vecn\trans\gamma \vece_d; q)$ is a Ramanujan sum.
Hence
\begin{equation}
\mathcal{E}_1
\ll \|F\|_{\Cb_b^0} \frac{1}{\#\cR_q} \sum_{\gamma\in \mathcal{B}_q} \gcd(q,\trans\vecn\trans\gamma \vece_d)
\ll \|F\|_{\Cb_b^0} \frac{1}{\#\cR_q} \sum_{\ell|q} \ell
\sum_{\substack{\gamma\in \mathcal{B}_q \\ \gcd(\trans\vecn\trans\gamma \vece_d, q)=\ell}}1.
\end{equation}

For each $\ell\mid q$, let
\begin{equation}
\scrS_\ell = \left\{\gamma\in \scrB_q\;:\; \gcd(\trans\vecn\trans\gamma \vece_d, q)=\ell\right\}.
\end{equation}
Then
\begin{equation}  \label{e:E1_ub}
\mathcal{E}_1 \ll \| F\|_{\Cb_b^0} \frac{1}{\#\cR_q} \sum_{\ell\mid q} \ell \#\scrS_\ell.
\end{equation}

Since $\vecn=\sm n_1\\ \vdots \\ n_d\esm \neq \vecnull$,
there exists $1\leq j_0 \leq d$ such that $n_{j_0}\neq 0$.
For $\gamma\in \scrS_\ell$, let $\veca = \trans\gamma \vece_d$ be the last row of $\gamma$.
Then
\begin{equation}
\trans\vecn(\trans\gamma \vece_d) = \trans\vecn \veca = n_1a_1+\cdots+n_da_d\equiv 0\pmod{\ell}
\end{equation}
and this implies that
\begin{equation}
n_{j_0} a_{j_0} \equiv -\sum_{1\leq j \leq d, j\neq j_0} n_j a_j\pmod{\ell}.
\end{equation}
Consequently $\gcd(n_{j_0}, \ell)\mid \sum_{1\leq j \leq d, j\neq j_0} n_j a_j$
and we get
\begin{equation}
a_{j_0}\equiv -\tilde{n}_{j_0}\frac{\sum_{1\leq j \leq d, j\neq j_0} n_j a_j}{\gcd(n_{j_0}, \ell)} \pmod{\ell/\gcd(n_{j_0}, \ell)},
\end{equation}
where $\tilde{n}_{j_0} \frac{n_{j_0}}{\gcd(n_{j_0}, \ell)}\equiv 1\pmod{\ell/\gcd(n_{j_0}, \ell)}$.
 We further note that, since $\gamma\in \SL_d(\Z)$, $\gcd(\veca)=1=\gcd(\veca, q)$.

For $\ell\mid q$, let $\ell_0= \gcd(n_{j_0}, \ell)$ and define
\begin{equation}
\scrA_{\ell_0}
= \left\{ {\veca\in (\Z/q\Z)^d \atop \gcd(\veca, q)=1} \;:\;
{\ell_0\mid \sum_{1\leq j \leq d, j\neq j_0}n_j a_j,
\atop a_{j_0}\equiv -\tilde{n}_{j_0}\frac{1}{\ell_0}\sum_{1\leq j \leq d, j\neq j_0} n_j a_j \pmod{\frac{\ell}{\ell_0}}}\right\}.
\end{equation}
By the above arguments, we deduce that for each $\gamma\in \scrS_\ell$, there exists $\veca\in \scrA_{\ell_0}$
such that $\veca \equiv \trans \gamma \vece_d \pmod{q}$.
Note that $(\Z/q\Z)^\times$ acts on the set $\scrA_{\ell_0}$ by scalar multiplication
and whenever $u\in (\Z/q\Z)^\times$, $\scrA_{\ell_0} = u\scrA_{\ell_0}$.
If $\trans\gamma \vece_d\equiv u \trans\gamma'\vece_d\pmod{q}$ for some $u\in (\Z/q\Z)^\times$,
then $\gamma'\in \Gamma_{0, d}(q) \gamma$.
So the map $\gamma \mapsto \trans \gamma \vece_d \pmod{q}$ is an injection into the set of orbits of $\scrA_{\ell_{0}}$ under the action of $(\Z/q\Z)^{\times}$ and hence we get
\begin{equation}
\# \scrS_\ell \leq \frac{1}{\varphi(q)} \#\scrA_{\ell_0}.
\end{equation}
The number of elements in $\scrA_{\ell_0}$ can, by its definition, be bounded as follows:
\begin{equation}
\#\scrA_{\ell_0} \leq q^{d-1} \frac{q}{\ell/\ell_0}.
\end{equation}
Combining both of the above inequalities,
\begin{equation} \label{e:Sl_ub}
\# \scrS_\ell \leq \frac{q^d}{\varphi(q)} \frac{\gcd(n_{j_0}, \ell)}{\ell}.
\end{equation}

For $n_{j_0}\neq 0$ and $|n_{j_0}| \leq \|\vecn\|_\infty \leq q^{\vartheta_1}$,
we have $\gcd(n_{j_0}, \ell) \leq q^{\vartheta_1}$ and inserting \eqref{e:Sl_ub} into \eqref{e:E1_ub}, we get
\begin{equation}
\scrE_1 \ll \|F\|_{\Cb_b^0} \frac{1}{\#\cR_q} \frac{q^d}{\varphi(q)} \sum_{\ell\mid q} \gcd(n_{j_0}, \ell)
\ll  \|F\|_{\Cb_b^0} \frac{1}{\#\cR_q} \frac{q^d}{\varphi(q)} q^{\vartheta_1} \sigma_0(q)
\ll_{d, \varepsilon}  \|F\|_{\Cb_b^0} q^{-1 + \vartheta_1 + \varepsilon},
\end{equation}
where we used \eqref{e:Rqlbd} for the last bound.
This completes the proof of \autoref{prop:E1}.
\end{proof}

\subsection{The second error term}

\begin{proof} [Proof of \autoref{prop:E2}]
By \eqref{e:bound_hatFn},
\begin{equation}
\mathcal{E}_2
\\ \leq
\|F\|_{\Cb_b^0}
\frac{1}{\#\cR_q}
\sum_{\gamma\in \mathcal B_q}
\sum_{\substack{\vecm\in \Z^d\setminus\{\vecnull\}, \\ \|\trans \gamma B_0 \vecm\|_\infty \leq q^{1-\frac{1}{d}+{\vartheta_2}}}}
\left|S(m_d, \trans\vecn\trans\gamma \vece_d; q)\right|.
\end{equation}
By Weil's bound for Kloosterman sums \cite{Estermann1961},
\begin{equation}
\left|S(m_d, \trans\vecn\trans\gamma \vece_d; q)\right|
\leq \sqrt{q} \gcd(m_d, \trans\vecn\trans\gamma \vece_d, q)^{\frac{1}{2}} \sigma_0(q)
\leq \sqrt{q} \gcd(m_d, q)^{\frac{1}{2}} \sigma_0(q).
\end{equation}
We thus have
\begin{equation}\label{e:errorbound_main_1}
\mathcal{E}_2
\\ \leq \|F\|_{\Cb_b^0}
\frac{1}{\#\cR_q}
\sum_{\gamma\in \mathcal B_q}
\sum_{\substack{\vecm\in \Z^d\setminus\{\vecnull\}, \\ \|\trans \gamma B_0 \vecm\|_\infty \leq q^{1-\frac{1}{d}+{\vartheta_2}}}}
\sqrt{q} \gcd(m_d, q)^{\frac{1}{2}} \sigma_0(q).
\end{equation}

For $\gamma\in \mathcal B_q$ and $\vecm\in \Z^d\setminus\{\vecnull\}$
with $\gcd(q, m_d)=\ell$,
we have
\begin{equation}
\trans\gamma B_0 \vecm = \trans\gamma \bpm qm_1\\ \vdots \\ qm_{d-1} \\ m_d\epm
=\ell \trans\gamma \bpm \frac{q}{\ell} m_1 \\ \vdots \\ \frac{q}{\ell} m_{d-1} \\ \frac{m_d}{\ell}\epm.
\end{equation}
Note that $\gcd(\frac{q}{\ell}, \frac{m_d}{\ell})=1$.
Set
\begin{equation}
\trans\gamma \frac{1}{\ell} B_0 \vecm = \vecx \in \Z^d.
\end{equation}
Since $\gamma\in \SL_d(\Z)$, $\vecx=\vecnull$ if and only if $\vecm=\vecnull$.
Assume that $\|\trans \gamma B_0 \vecm\|_\infty \leq q^{1-\frac{1}{d}+{\vartheta_2}}$.
Then $\vecx\in \Z^d\setminus\{\vecnull\}$ and $\|\vecx\|_\infty \leq \frac{q^{1-\frac{1}{d}+{\vartheta_2}}}{\ell}$.
Moreover since $\|\vecx\|_\infty <1$ if and only if $\vecx=\vecnull$,
we only consider $\ell\mid q$ such that $\frac{q^{1-\frac{1}{d}+{\vartheta_2}}}{\ell}\geq 1$.

Summarising, for each given $\vecx\in \Z^d\setminus\{\vecnull\}$ with $\|\vecx\|_\infty \leq \frac{q^{1-\frac{1}{d}+{\vartheta_2}}}{\ell}$, we count the number of $\gamma\in \mathcal B_q$ such that
$\trans\gamma \vecm = \vecx$ has an integral solution $\vecm\in \Z^d$ satisfying
$\frac{q}{\ell}\mid m_j$ for $1\leq j \leq d-1$ and $\gcd(\frac{q}{\ell}, m_d)=1$.
Moreover the solution $\vecm$ is uniquely determined since $\vecm = \trans\gamma^{-1} \vecx$.
So we can write
\begin{multline} \label{eqn:m2x}
\frac{1}{\#\cR_q}
\sum_{\gamma\in \mathcal B_q}
\sum_{\substack{\vecm\in \Z^d\setminus\{\vecnull\}, \\ \|\trans \gamma B_0 \vecm\|_\infty \leq q^{1-\frac{1}{d}+{\vartheta_2}}}}
\sqrt{q} \gcd(m_d, q)^{\frac{1}{2}} \sigma_0(q)
\\ =
\frac{1}{\#\cR_q}
\sqrt{q} \sigma_0(q)
\sum_{\substack{\ell\mid q, \\ \ell \leq q^{1-\frac{1}{d}+{\vartheta_2}}}}\ell^{\frac{1}{2}}
\sum_{\substack{\vecx\in \Z^d\setminus\{\vecnull\}, \\ \|\vecx\|_\infty \leq \frac{q^{1-\frac{1}{d}+{\vartheta_2}}}{\ell}}}
\sum_{\gamma\in \mathcal B_q} \sum_{\substack{\vecm\in \Z^d\setminus\{\vecnull\}, \\ \frac{q}{\ell}\mid m_j, 1\leq j \leq d-1, \\ \gcd(m_d, \frac{q}{\ell})=1, \\ \trans\gamma \vecm=\vecx}} 1.
\end{multline}

For $\ell\mid q$ satisfying $\ell\leq q^{1-\frac{1}{d}+{\vartheta_2}}$, and
for each $\vecx\in \Z^d \setminus\{\vecnull\}$ with $\|\vecx\|_\infty \leq q^{1-1/d+\vartheta_2} /\ell$, let
\begin{equation}
\scrS_{\ell}(\vecx)
= \left\{\gamma\in \mathcal B_q\;:\;
{\exists  \vecm\in \Z^d \textrm{ satisfies } \trans\gamma \vecm =\vecx , \atop \frac{q}{\ell}\mid m_j, \; 1\leq j \leq d-1, \; \gcd(m_d, q/\ell)=1}\right\}.
\end{equation}
Then \eqref{eqn:m2x} is equal to
\begin{equation}
\frac{1}{\#\cR_q}
\sqrt{q} \sigma_0(q)
\sum_{\substack{\ell\mid q, \\ \ell \leq q^{1-\frac{1}{d}+{\vartheta_2}}}}\ell^{\frac{1}{2}}
\sum_{\substack{\vecx\in \Z^d\setminus\{\vecnull\}, \\ \|\vecx\|_\infty \leq \frac{q^{1-\frac{1}{d}+{\vartheta_2}}}{\ell}}}
\#\scrS_\ell(\vecx).
\end{equation}

We claim that
\begin{equation}\label{e:cardinality_scrSell}
\#\scrS_\ell(\vecx)\leq
\frac{[\SL_d(\Z): \Gamma_{0, d}(q)]}{[\SL_d(\Z): \Gamma_{0, d}(q/\ell)]}
= \ell^{d-1}\prod_{p\mid \ell, p\nmid q/\ell}\frac{1-p^{-d}}{1-p^{-1}}.
\end{equation}
Indeed, for $\vecx \in \Z^d \setminus \{ \vecnull \}$,
consider $\gamma$ and $\widetilde{\gamma}$ in $\SL_d(\Z)$ such that
there exist $\vecm$ and $\vecn$ satisfying:
$\trans \gamma \vecm = \vecx$, $\trans \widetilde{\gamma} \vecn = \vecx$
with $\frac{q}{\ell} \mid m_i$ and $\frac{q}{\ell} \mid n_i$ for $1 \le i \le d-1$,
while $\gcd(\frac{q}{\ell}, m_d) = \gcd(\frac{q}{\ell}, n_d) = 1$.
It follows that
\begin{equation}
\vecm = \trans \gamma^{-1} \vecx = \trans \gamma^{-1} \trans \widetilde{\gamma} \vecn = \trans (\widetilde{\gamma} \gamma^{-1}) \vecn .
\end{equation}
Upon reducing modulo $\frac{q}{\ell}$, we get
\begin{equation}
\trans (\widetilde{\gamma} \gamma^{-1}) \bpm \vecnull \\ n_d \epm \equiv \bpm \vecnull \\ m_d \epm \pmod{\frac{q}{\ell}}
\end{equation}
with $n_d$ and $m_d$ both invertible modulo $\frac{q}{\ell}$.
This means $\widetilde{\gamma} \gamma^{-1} \in \Gamma_{0,d}(q/\ell)$.
Now fix one matrix $\gamma\in\SL_d(\Z)$ satisfying the condition in the definition of $\mathcal{S}_\ell(\vecx)$
(if no such $\gamma$ exists, then $\# \scrS_\ell(\vecx)=0$ and the claim is proved).
Every matrix $\widetilde \gamma\in \scrS_\ell(\vecx)$
is then of the form $\widetilde{\gamma} = \delta \gamma$
for some $\delta\in \Gamma_{0,d}(q/\ell)$.
Hence  $\# \scrS_\ell(\vecx)$ is bounded by the number of distinct $\Gamma_{0,d}(q)$-cosets of the form $\Gamma_{0,d}(q)\delta \gamma$ with $\delta\in \Gamma_{0,d}(q/\ell)$.
One can check that for any $\delta,\delta'\in \Gamma_{0,d}(q/\ell)$, we have $\Gamma_{0,d}(q)\delta \gamma = \Gamma_{0,d}(q) \delta' \gamma$ if and only if $\delta' \delta^{-1} \in \Gamma_{0,d}(q)$.
Therefore
\begin{equation}
  \# \scrS_\ell(\vecx) \leq \#(\Gamma_{0,d}(q) \backslash \Gamma_{0,d}(q/\ell)) = [\Gamma_{0,d}(q/\ell) : \Gamma_{0,d}(q)],
\end{equation}
which is precisely the inequality in \eqref{e:cardinality_scrSell}.
The equality follows from \autoref{prop:Gamma0_index}.

Then we have
\begin{align}
\frac{1}{\#\cR_q}
\sqrt{q} \sigma_0(q) &
\sum_{\ell\mid q}\ell^{\frac{1}{2}}
\sum_{\substack{\vecx\in \Z^d\setminus\{\vecnull\}, \\ \|\vecx\|_\infty \leq \frac{q^{1-\frac{1}{d}+{\vartheta_2}}}{\ell}}}
\#\scrS_\ell(\vecx)
\\ &\leq
\frac{1}{\#\cR_q}
\sqrt{q} \sigma_0(q)
\sum_{\substack{\ell\mid q, \\ \ell \leq q^{1-\frac{1}{d}+{\vartheta_2}}}}\ell^{\frac{1}{2}}
\sum_{\substack{\vecx\in \Z^d\setminus\{\vecnull\}, \\ \|\vecx\|_\infty \leq \frac{q^{1-\frac{1}{d}+{\vartheta_2}}}{\ell}}} \ell^{d-1} \prod_{p\mid \ell, p\nmid q/\ell} \frac{1-p^{-d}}{1-p^{-1}}
\nonumber \\
& \ll
\frac{1}{\#\cR_q}
\sqrt{q} \sigma_0(q)\prod_{p\mid q}\frac{1-p^{-d}}{1-p^{-1}}
\sum_{\substack{\ell\mid q, \\ \ell \leq q^{1-\frac{1}{d}+{\vartheta_2}}}}\ell^{\frac{1}{2}}
\left(\frac{q^{1-\frac{1}{d}+{\vartheta_2}}}{\ell}\right)^d\ell^{d-1}
\nonumber \\ & = \frac{1}{\#\cR_q}
q^{d-\frac{1}{2}+{\vartheta_2} d} \sigma_0(q)\prod_{p\mid q}\frac{1-p^{-d}}{1-p^{-1}}
\sum_{\substack{\ell\mid q, \\ \ell \leq q^{1-\frac{1}{d}+{\vartheta_2}}}}
\ell^{-\frac{1}{2}}
\nonumber \\ & \leq \frac{1}{\#\cR_q}
q^{d-\frac{1}{2}+{\vartheta_2} d} \sigma_0(q)^2 \prod_{p\mid q}\frac{1-p^{-d}}{1-p^{-1}}.
\nonumber
\end{align}

Using \eqref{e:RqEuler}, we get
\begin{equation}
\frac{1}{\#\cR_q}
\sqrt{q} \sigma_0(q)
\sum_{\ell\mid q}\ell^{\frac{1}{2}}
\sum_{\substack{\vecx\in \Z^d\setminus\{\vecnull\}, \\ \|\vecx\|_\infty \leq \frac{q^{1-\frac{1}{d}+{\vartheta_2}}}{\ell}}}
\#\scrS_\ell(\vecx)
\leq
q^{-\frac{1}{2}+{\vartheta_2} d}
\frac{\sigma_{0}(q)^2}{\prod_{p\mid q}(1-p^{-1})}.
\end{equation}
This proves \autoref{prop:E2}.
\end{proof}

\subsection{The third error term}


\begin{proof}[Proof of \autoref{prop:E3}]
By \eqref{e:bound_hatFn}, for any integer $k\geq0$, we have
\begin{equation}
\mathcal{E}_3
\\ \leq
\frac{1}{\#\cR_q}
\sum_{\gamma\in \mathcal B_q
}
\sum_{\substack{\vecm\in \Z^d\setminus\{\vecnull\}, \\ \|\trans \gamma B_0 \vecm\|_\infty > q^{1-\frac{1}{d}+{\vartheta_2}}}} \frac{\|F\|_{\Cb_b^k}\left|S(m_d, \trans\vecn(\trans\gamma \vece_d); q)\right|}{(2\pi \|q^{-1+\frac{1}{d}} \trans\gamma B_0 \vecm\|_{\infty})^k}
.
\end{equation}
By the trivial bound for the Kloosterman sum
$
\left|S(m_d, \trans\vecn(\trans\gamma \vece_d); q)\right|
\leq \varphi(q),
$
we have
\begin{equation}
\mathcal{E}_3
\leq
\|F\|_{\Cb_b^k} \frac{\varphi(q)}{(2\pi)^k}
\frac{1}{\#\cR_q}
\sum_{\gamma\in \mathcal B_q}
\sum_{\substack{\vecm\in \Z^d\setminus\{\vecnull\}, \\ \|\trans \gamma B_0 \vecm\|_\infty > q^{1-\frac{1}{d}+{\vartheta_2}}}} \frac{q^{k(1-\frac{1}{d})}}{\|\trans\gamma B_0 \vecm\|_{\infty}^k}.
\end{equation}

For each $\gamma\in \mathcal B_q$
and $\vecm\in \Z^d\setminus\{\vecnull\}$, let $\trans \gamma B_0 \vecm = \vecx$, then we have
\begin{multline}\label{eqn:E3<<Sum}
\scrE_3
\leq \|F\|_{\C_b^k} \frac{1}{(2\pi)^k}
\frac{\varphi(q) \#{\scrB_q}}{\#\cR_q}
\sum_{\substack{\vecx\in \Z^d\setminus\{\vecnull\}, \\ \|\vecx\|_\infty > q^{1-\frac{1}{d}+{\vartheta_2}}}}
\frac{q^{k(1-\frac{1}{d})}}{\|\vecx\|_\infty^{k-d-{\vartheta_2}}} \frac{1}{\|\vecx\|_\infty^{d+{\vartheta_2}}}
\\
\leq \|F\|_{\C_b^k} \frac{1}{(2\pi)^k}
q^{d-1-{\vartheta_2}(k-d-1+\frac{1}{d}-{\vartheta_2})}
\sum_{\substack{\vecx\in \Z^d\setminus\{\vecnull\}, \\ \|\vecx\|_\infty > q^{1-\frac{1}{d}+{\vartheta_2}}}}
\frac{1}{\|\vecx\|_\infty^{d+{\vartheta_2}}}.
\end{multline}
By the same argument used to obtain \eqref{Gauss}, we have
\begin{equation}
\sum_{\substack{\vecx\in \Z^d\setminus\{\vecnull\}, \\ \|\vecx\|_\infty > q^{1-\frac{1}{d}+{\vartheta_2}}}}
\frac{1}{\|\vecx\|_\infty^{d+{\vartheta_2}}} \ll_{d,{\vartheta_2}}
(q^{1-\frac{1}{d}+{\vartheta_2}})^{-{\vartheta_2}}.
\end{equation}
Thus, by \eqref{eqn:E3<<Sum}, we have
\[
  \scrE_3
  \ll_{d,{\vartheta_2},k}  \|F\|_{\C_b^k}
  q^{d-1-{\vartheta_2}(k-d)}
  \ll_{d,{\vartheta_2},k}  \|F\|_{\C_b^k}
  q^{-\frac{1}{2}+{\vartheta_2} d},
\]
provided that $k\geq \frac{2d-1}{2{\vartheta_2}}$.
This proves \autoref{prop:E3}.
\end{proof}

\section{An application: diameters of random circulant graphs} \label{sec:app}
In this section, we denote by $X$ the space of unimodular lattices in $\R^d$.

We abuse notations and still denote by $\Cb_b^k(X)$ the space of $k$-times continuously differentiable functions $f$ from $X$ to $\R$ such that for every left-invariant differential operator $D$ on $\SL_d(\R)$ of order at most $k$, $\| D f \|_{\infty}$ is finite.
Likewise, for a function $f \in \Cb_b^k(X)$, we still denote by $\| f \|_{\Cb_b^k}$ the obvious analogue of \eqref{e:FCbk}.

Define, for $q \ge 2$ and $d \ge 2$, the $(d+1)$-dimensional lattice $\Lambda_q = \Z^d \times q \Z$.
For $\veca \in (\Z \cap [1, q])^d$ with $\gcd(\veca, q) = 1$ (meaning $\veca \in \cR_q$),
define \begin{equation} n(\veca) = \bpm I_d & \veca \\ \trans \vecnull & 1 \epm \in \SL_{d+1}(\Z). \end{equation}
Consider $\Lambda_q(\veca)_0 = \Lambda_q n(\veca) \cap (\R^d \times \{0\})$.
Finally define $D_q = q^{-1/d} I_d \in \GL_d(\R)$, so that $\det(D_q) = q^{-1}$.
Consider the $d$-dimensional lattice $L_{q, \veca} = \Lambda_q(\veca)_0 D_q$.
Following the steps used to prove \cite[Theorem 3]{MarklofStrombergsson2013}, with \autoref{thm:main} replacing the use of \cite[Theorem 4]{MarklofStrombergsson2013}, we see that \autoref{thm:main} implies:

\begin{theorem} \label{thm:efflatteq}
For every $d \ge 3$, every $\epsilon>0$ and every function $f \in \Cb_b^{k}(X)$, with an integer  $k \geq 2d^2-d+1$, we have
\begin{equation}
\frac 1{\#\cR_q} \sum_{\veca \in \cR_q} f(L_{q, \veca}) = \int_X f d\mu + O(\|f\|_{\Cb_b^{k}}  q^{-\frac 12 + \frac{d^2(2k-2d+1)}{2k^2}+\epsilon}).
\end{equation}
\end{theorem}

\begin{remark}
When $d=2$, a version of this theorem follows from \cite[Theorem 1.3]{LeeMarklof2017} (see also \autoref{rem:d=2}).
\end{remark}

As explained in the introduction, we can use this theorem to deduce the following rate of convergence for the limiting distribution of the appropriately rescaled diameters of random circulant graphs.

To help understand what follows, we briefly summarise the key steps in the relevant parts of Marklof and Str\"ombergsson's paper \cite{MarklofStrombergsson2013}.
The first step (see \cite[section 2.2]{MarklofStrombergsson2013} for more details) is to identify the circulant graph $C_q(\veca)$ --- that is, the Cayley graph of $\Z/q\Z$ with respect to the $a_i$ --- with a lattice graph on a torus:
\begin{enumerate}
\item consider the graph $LG_d$ whose vertices are the points of the lattice $\Z^d$ and whose edges are of the form $(\veck, \veck + \vece_i)$ for some $\veck \in \Z^d$, where $(\vece_1, \ldots, \vece_d)$ is the canonical basis of $\R^d$;
\item introduce a metric $m$ on $LG_d$ by defining the distance between two vertices $\veck$ and $\vecl$ in $\Z^d$ to be $m(\veck, \vecl) = \sum_{i=1}^d |k_i - l_i|$
\item extend this metric in the obvious way to a metric on $\Z^d/\Lambda$ where $\Lambda$ is a sublattice of $\Z^d$;
\item \cite[Lemma 2]{MarklofStrombergsson2013} is the assertion that $LG_d/\Lambda_q(\veca)_0$ and $C_q(\veca)$ are isomorphic as metric graphs.
\end{enumerate}
The next step is to relate the diameter of $LG_d/\Lambda_q(\veca)_0$ --- which, by the first step, is exactly the diameter $\diam(q,d)$ we are interested in --- to the diameter of $\R^d/L_{q,\veca}$ (where the distance on the torus is the $\ell^1$ distance):
\cite[Proposition 1]{MarklofStrombergsson2013} asserts that
\begin{equation}
q^{1/d} \diam(\R^d/L_{q,\veca}) - \frac d2 \le \diam(LG_d/\Lambda_q(\veca)_0) \le q^{1/d} \diam(\R^d/L_{q,\veca}).
\end{equation}
The final step (\cite[Lemma 4]{MarklofStrombergsson2013}) connects the diameter of a torus $\R^d/L$ to the covering radius of the $d$-orthoplex with respect to the lattice $L \subset \R^d$:
\begin{equation}
\diam(\R^d/L) = \rho(\fP, L).
\end{equation}
We recall that the latter quantity is defined to be
\begin{equation}
\rho(\fP, L) = \inf \{r > 0 \, : \, r \fP + L = \R^d \}
\end{equation}
and that for $d \ge 2$, the $d$-orthoplex is the polytope
\begin{equation} \fP = \{ \vecx \in \R^d \, : \, \| \vecx \|_1 \le 1 \}. \end{equation}

We can now state the consequence of \autoref{thm:efflatteq} pertaining to the diameters of random circulant graphs.

\begin{corollary} \label{cor:app_err}
For every $d \ge 3$, there exists a continuous non-increasing function $\Psi_d \colon \R_{\ge 0} \to \R_{\ge 0}$
with $\Psi_d(0)=1$ and $\lim_{R \to \infty} \Psi_d(R) = 0$
such that for every $\epsilon > 0$ and every $R \ge 0$, we have
\begin{equation}
\Prob\left(\frac{\diam(q,d)}{q^{1/d}} \ge R \right) = \Psi_d(R) + O_{R,\epsilon} \left(q^{-\eta_d+\epsilon} \right),
\end{equation}
where $\eta_d = \dfrac {2d^2 - 2d + 1}{2 (2d^2 - d + 1)^2 (2d^2 - d + 2)}$.
Moreover, for $R \ge 0$, $\Psi_d$ is explicitly given by
\begin{equation}
\Psi_d(R) = \mu(\{ L \in X \, : \, \rho(\fP, L) \ge R \})
\end{equation}
where $\mu$ is the Haar probability measure on $X$.
\end{corollary}

It should be clear from the discussion preceding the above corollary that its proof requires an approximation argument to pass from the smooth functions in \autoref{thm:efflatteq} to characteristic functions.
We borrow the following definition from Li's paper (\cite[Definition 1.3]{Li2015}):
\begin{definition} \label{def:thinbd} A subset of $X$ is said to have \emph{thin boundary} if its boundary is contained in the union of finitely many connected smooth submanifolds of $X$, all of which have codimension at least $1$.
\end{definition}

We also borrow (in a slightly modified form) the following technical lemma from a paper by Str\"ombergsson and Venkatesh (\cite[Lemma 1]{StrombergssonVenkatesh2005}). For a set $S \subset X$, we denote by $\chi_S \colon X \to \{0, 1\}$ its characteristic function.
\begin{lemma} \label{lem:SVapprox}
If $S \subset X$ has thin boundary, then for each $\delta \in (0,1)$, there exist functions $f_{-}$ and $f_{+}$ in $C^\infty(X)$ such that for every $k \ge 1$,
\begin{enumerate}
\item $0 \le f_{-} \le \chi_S \le f_{+} \le 1$;
\item $\|f_{-}\|_{\Cb_b^k} \ll_S \delta^{-k}$ and $\|f_{+}\|_{\Cb_b^k} \ll_S \delta^{-k}$;
\item $\| f_{-} - \chi_S \|_{L^1} \ll_S \delta$ and $\| f_{+} - \chi_S \|_{L^1} \ll_S \delta$.
\end{enumerate}
\end{lemma}

We can finally proceed with the proof of \autoref{cor:app_err}.
\begin{proof}[Proof of \autoref{cor:app_err}]
Define, for $R \ge 0$, the following subset of $d$-dimensional unimodular lattices
\begin{equation}
S_R = \{ L \in X \, : \, \rho(\fP, L) \ge R \}
\end{equation}
where $\rho$ is the covering radius and $\fP$ is the $d$-orthoplex.

In order to deduce \autoref{cor:app_err}, we wish to apply \autoref{thm:efflatteq} to $\chi_{S_R}$ for each $R \ge 0$.
To do so, we make use of \autoref{lem:SVapprox} to approximate this characteristic function by smooth functions.
For this, we first need to show that, for each $R \ge 0$, the set $S_R$ has thin boundary according to \autoref{def:thinbd}.
However, this follows from the proof of \cite[Lemma 7]{MarklofStrombergsson2013}.
We therefore find smooth functions $f_{-}$ and $f_{+}$ as in \autoref{lem:SVapprox}.
Applying \autoref{cor:app_err} to each of those and using their properties, we conclude that for every $\delta \in (0,1)$, every $\epsilon > 0$ and every $k \ge 2d^2-d+1$,
\begin{equation}
\frac 1{\#\cR_q} \sum_{\veca \in \cR_q} \chi_{S_R}(L_{q, \veca}) = \int_X \chi_{S_R} d\mu + O_R(\delta + \delta^{-k} q^{-\frac 12 + \vartheta +\epsilon})
\end{equation}
with $\vartheta=\frac{d^2(2k-2d+1)}{2k^2}$.
If we now choose $\delta = q^{\frac {-1/2 + \vartheta}{k+1}}$, we get that for every $\epsilon > 0$,
\begin{equation}
\frac 1{\#\cR_q} \sum_{\veca \in \cR_q} \chi_{S_R}(L_{q, \veca}) = \int_X \chi_{S_R} d\mu + O_R(q^{\kappa_{d,k} +\epsilon})
\end{equation}
with $\kappa_{d,k} = \frac {-k^2+d^2(2k-2d+1)}{2k^2(k+1)}$.
Finally, picking $k = 2d^2-d+1$ we get the desired error term with
$\eta_d \sim \frac 1{8d^4}$ as claimed.
\end{proof}

\thispagestyle{empty}
{\footnotesize
\nocite{*}
\bibliographystyle{amsalpha}
\bibliography{bibliography}
}

\end{document}